\documentclass{article} 

\usepackage{tikz}
\tikzset{every picture/.style={line width=0.75pt}} %set default line width to 0.75pt   

 \usepackage[usenames,dvipsnames]{pstricks}
 \usepackage{epsfig}
 \usepackage{pst-grad} % For gradients
 \usepackage{pst-plot} % For axes
 \usepackage[space]{grffile} % For spaces in paths
 \usepackage{etoolbox} % For spaces in paths
 \makeatletter % For spaces in paths
 \patchcmd\Gread@eps{\@inputcheck#1 }{\@inputcheck"#1"\relax}{}{}
 \makeatother

\usepackage[T1]{fontenc} 

 \usepackage[usenames,dvipsnames]{pstricks}
 \usepackage{epsfig}
 \usepackage{pst-grad} % For gradients
 \usepackage{pst-plot} % For axes
 \usepackage[space]{grffile} % For spaces in paths
 \usepackage{etoolbox} % For spaces in paths
 \makeatletter % For spaces in paths
 \patchcmd\Gread@eps{\@inputcheck#1 }{\@inputcheck"#1"\relax}{}{}
 \makeatother

\usepackage{amsmath}
\usepackage{amssymb}
\usepackage{amsthm}
\usepackage[latin1]{inputenc}
\usepackage{graphicx}

\usepackage[usenames,dvipsnames]{pstricks}
\usepackage{epsfig}
\usepackage{pst-grad} % For gradients
\usepackage{pst-plot} % For axes

\usepackage[misc,geometry]{ifsym}

\usepackage{ifthen}
\usepackage{color}

\newcommand{\intav}[1]{\mathchoice {\mathop{\vrule width 6pt height 3 pt depth  -2.5pt
\kern -8pt \intop}\nolimits_{\kern -6pt#1}} {\mathop{\vrule width
5pt height 3  pt depth -2.6pt \kern -6pt \intop}\nolimits_{#1}}
{\mathop{\vrule width 5pt height 3 pt depth -2.6pt \kern -6pt
\intop}\nolimits_{#1}} {\mathop{\vrule width 5pt height 3 pt depth
-2.6pt \kern -6pt \intop}\nolimits_{#1}}}

\def\polhk#1{\setbox0=\hbox{#1}{\ooalign{\hidewidth\lower1.5ex\hbox{`}\hidewidth\crcr\unhbox0}}}

\def\XXint#1#2#3{{\setbox0=\hbox{$#1{#2#3}{\int}$ }
\vcenter{\hbox{$#2#3$ }}\kern-.6\wd0}}

 \newcommand{\Rr}{\mathbb R}

\newcommand{\dist}{\operatorname{dist}}

\usepackage{setspace}

\newtheorem{theorem}{Theorem}

\newtheorem{definition}{Definition}
\newtheorem{lemma}{Lemma}
\newtheorem{corollary}{Corollary}
\newtheorem{proposition}{Proposition}
\newtheorem{remark}{Remark}
\newtheorem{Assumption}{A}

\makeatletter
\patchcmd{\env@cases}{1.2}{.8}{}{}
\makeatother

\begin{document}

\title{Boundary regularity for a fully nonlinear free transmission problem}
\author{David Jesus, Edgard A. Pimentel and David Stolnicki}
	
\date{\today} %%

\maketitle

\begin{abstract}

\noindent We examine boundary regularity for a fully nonlinear free transmission problem. We argue using approximation methods, comparing the operators driving the problem with a limiting profile. Working natural conditions on the data of the problem, we produce regularity estimates in Sobolev and  $C^{1,{\rm Log-Lip}}$-spaces. Our findings extend recent developments in the literature to the free boundary setting.

\medskip


\medskip 

\end{abstract}

\vspace{.1in}

\section{Introduction}\label{sec_intro}

We examine a fully nonlinear free transmission problem of the form
\begin{equation}\label{eq_main}
	\begin{split}
		F_1(D^2u,Du,x)=f&\hspace{.2in}\mbox{in}\hspace{.2in}\Omega\cap\left\lbrace u>0\right\rbrace\\
		F_2(D^2u,Du,x)=f&\hspace{.2in}\mbox{in}\hspace{.2in}\Omega\cap\left\lbrace u<0\right\rbrace,
	\end{split}
\end{equation}
equipped with the Dirichlet boundary condition $u=g$ on $\partial\Omega$. We suppose $F_i:S(d)\times\mathbb{R}^d\times\Omega\to\mathbb{R}$ satisfies a usual structure condition, for $i=1,\,2$. Also, $f\in L^p(\Omega)$ for some $p_0<p$, where $d/2<p_0$ is the integrability exponent above which the Aleksandrov-Bakelman-Pucci estimate is available for the solutions to $F=f\in L^p(\Omega)$. The boundary datum $g$ belongs to suitable functional spaces to be set further. We denote with $\Gamma(u)$ the free boundary given by
\[
	\Gamma(u):=\partial\left\lbrace u>0\right\rbrace \cup \partial\left\lbrace u<0\right\rbrace.
\]

Our findings comprise boundary regularity in spaces $C^{1,{\rm Log-Lip}}$ and boundary $W^{2,p}$-estimates. We emphasise these two classes of results are different. The first one concerns estimates for points $x_0\in\partial\Omega$. The latter ensures the existence of a universal neighbourhood $\Omega_\delta$ of $\partial\Omega$ such that $u\in W^{2,p}(\Omega_\delta)$. We stress that our findings do not rely on convexity or convexity-like assumptions on the operators $F_i$.

Transmission problems were introduced in the seminal work of Mauro Picone \cite{Picone1954}, motivated by a model in the theory of elasticity. Consequential, Picone's work attracted relevant attention, as several authors covered important aspects of the problem and expanded its scope; we refer the reader to \cite{Stampacchia1956,Lions1956, Campanato1957,Campanato1959,Campanato1959a,Schechter1960}. The contributions mentioned above focus on the existence and uniqueness of solutions as well as on a priori estimates. The regularity theory associated with transmission problems started to be examined circa 2000. See \cite{Li-Vogelius2000,Li-Nirenberg2003}; see also \cite{Bao-Li-Yin1,Bao-li-Yin2}. In those papers, the authors consider models driven by equations in the divergence form, in the scalar and vectorial settings. Their findings cover gradient estimates with applications to the study of insulation, conductivity and composite materials. 

The issue of regularity theory for transmission problems has attracted renewed interest. Mostly motivated by the questions and methods put forward in \cite{CSCS2021}. In that paper, the authors examine a transmission problem governed by the Laplace operator in the presence of a $C^{1,\alpha}$-regular interface. Arguing through mean-value formulas for harmonic functions and new stability results for close-to-flat interfaces, the authors prove that solutions are of class $C^{1,\alpha}$ up to the interface.

The program launched in \cite{CSCS2021} is pursued in the fully nonlinear context in \cite{SoriaCarro-Stinga}. Here, the authors prove a new Aleksandrov-Bakelman-Pucci maximum principle. They establish estimates for the solutions to fully nonlinear transmission problems up to the interface. The solutions' regularity matches the interface's regularity in $C^{k,\alpha}$-spaces, for $k=0,1,2$. 

In the aforementioned developments, the transmission interface is given and fixed a priori. Conversely, one could consider a model with solution-dependent, or \emph{free}, interfaces. For instance, in the region where solutions are positive, the diffusion process is governed by an operator $F_1$. However, an operator $F_2$ drives the model in the region where the solutions are negative.

Free transmission problems are mathematical models describing discontinuous diffusions. Such discontinuities are solution-dependent, giving rise to a free boundary. The discontinuous dependence on $u$ entails genuine difficulties in the analysis of \eqref{eq_main}. As a consequence, the existence and local regularity of the solutions require new methods and techniques; see \cite{Pimentel-Swiech2022,Pimentel-Santos2023}. For related developments in the degenerate fully nonlinear setting, see \cite{Huaroto-Pimentel-Rampasso-Swiech}; see also \cite{DJ1,DJ2,CdF}.

We examine the boundary regularity for $L^p$-viscosity solutions to \eqref{eq_main}. Our focus lies on points at the intersection of $\partial\Omega$ with the free boundary $\Gamma(u)$. Indeed, if $x_0\in \partial\Omega$ but $x_0\notin \Gamma(u)$, the regularity of solutions in a neighbourhood of $x$ is well-understood and documented; see, for instance, \cite{SilSir_2014}. Moreover, if we take $x_1\in\partial\Omega\cap\partial\left\lbrace u>0\right\rbrace$ with $x_1\notin \partial\left\lbrace u<0\right\rbrace$, the problem in \eqref{eq_main} reduces to a one-phase obstacle problem. 

Suppose however that $x_2\in\partial\Omega\cap \partial\left\lbrace u>0\right\rbrace\cap\partial\left\lbrace u<0\right\rbrace$. In this case, a new strategy is required to produce regularity results in the vicinity of $x_2$ in $\Omega^+$. Our analysis, though more general, addresses that case in particular; see Figure \ref{fig_intro}.

\bigskip

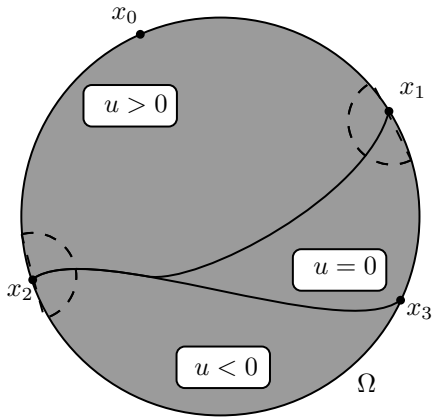
\begin{figure}[h]
\centering

\begin{tikzpicture}[x=0.55pt,y=0.55pt,yscale=-1,xscale=1]

%Shape: Circle [id:dp5644438575186321] 
\draw  [fill={rgb, 255:red, 155; green, 155; blue, 155 }  ,fill opacity=0.23 ] (62,150.91) .. controls (62,75.3) and (123.3,14) .. (198.91,14) .. controls (274.53,14) and (335.83,75.3) .. (335.83,150.91) .. controls (335.83,226.53) and (274.53,287.83) .. (198.91,287.83) .. controls (123.3,287.83) and (62,226.53) .. (62,150.91) -- cycle ;
%Curve Lines [id:da3360303677545764] 
\draw    (69.22,194.31) .. controls (109.22,164.31) and (282.83,238.52) .. (322.83,208.52) ;
%Curve Lines [id:da5543414450647156] 
\draw    (69.22,194.31) .. controls (83.21,183.82) and (113.51,186.07) .. (148.73,192.15) .. controls (183.94,198.22) and (299.87,137.3) .. (314.87,78.3) ;
%Shape: Circle [id:dp5693959539487511] 
\draw  [fill={rgb, 255:red, 0; green, 0; blue, 0 }  ,fill opacity=1 ] (320.48,208.52) .. controls (320.48,207.23) and (321.53,206.18) .. (322.83,206.18) .. controls (324.12,206.18) and (325.18,207.23) .. (325.18,208.52) .. controls (325.18,209.82) and (324.12,210.87) .. (322.83,210.87) .. controls (321.53,210.87) and (320.48,209.82) .. (320.48,208.52) -- cycle ;
%Shape: Circle [id:dp21064006689416226] 
\draw  [fill={rgb, 255:red, 0; green, 0; blue, 0 }  ,fill opacity=1 ] (312.48,78.52) .. controls (312.48,77.23) and (313.53,76.18) .. (314.83,76.18) .. controls (316.12,76.18) and (317.18,77.23) .. (317.18,78.52) .. controls (317.18,79.82) and (316.12,80.87) .. (314.83,80.87) .. controls (313.53,80.87) and (312.48,79.82) .. (312.48,78.52) -- cycle ;
%Shape: Circle [id:dp08308752928057639] 
\draw  [fill={rgb, 255:red, 0; green, 0; blue, 0 }  ,fill opacity=1 ] (67.48,194.52) .. controls (67.48,193.23) and (68.53,192.18) .. (69.83,192.18) .. controls (71.12,192.18) and (72.18,193.23) .. (72.18,194.52) .. controls (72.18,195.82) and (71.12,196.87) .. (69.83,196.87) .. controls (68.53,196.87) and (67.48,195.82) .. (67.48,194.52) -- cycle ;
%Rounded Rect [id:dp30681177460319686] 
\draw  [fill={rgb, 255:red, 255; green, 255; blue, 255 }  ,fill opacity=1 ] (104.74,66.14) .. controls (104.74,62.99) and (107.3,60.43) .. (110.46,60.43) -- (163.29,60.43) .. controls (166.44,60.43) and (169,62.99) .. (169,66.14) -- (169,83.29) .. controls (169,86.44) and (166.44,89) .. (163.29,89) -- (110.46,89) .. controls (107.3,89) and (104.74,86.44) .. (104.74,83.29) -- cycle ;
%Rounded Rect [id:dp38998939226041196] 
\draw  [fill={rgb, 255:red, 255; green, 255; blue, 255 }  ,fill opacity=1 ] (248.74,179.14) .. controls (248.74,175.99) and (251.3,173.43) .. (254.46,173.43) -- (307.29,173.43) .. controls (310.44,173.43) and (313,175.99) .. (313,179.14) -- (313,196.29) .. controls (313,199.44) and (310.44,202) .. (307.29,202) -- (254.46,202) .. controls (251.3,202) and (248.74,199.44) .. (248.74,196.29) -- cycle ;
%Rounded Rect [id:dp7136339688574889] 
\draw  [fill={rgb, 255:red, 255; green, 255; blue, 255 }  ,fill opacity=1 ] (168.74,246.14) .. controls (168.74,242.99) and (171.3,240.43) .. (174.46,240.43) -- (227.29,240.43) .. controls (230.44,240.43) and (233,242.99) .. (233,246.14) -- (233,263.29) .. controls (233,266.44) and (230.44,269) .. (227.29,269) -- (174.46,269) .. controls (171.3,269) and (168.74,266.44) .. (168.74,263.29) -- cycle ;
%Shape: Arc [id:dp6534670806649483] 
\draw  [draw opacity=0][dash pattern={on 4.5pt off 4.5pt}] (62.17,163.17) .. controls (77.4,159.13) and (93.34,167.63) .. (98.33,182.8) .. controls (103.5,198.54) and (94.94,215.5) .. (79.2,220.67) .. controls (78.78,220.81) and (78.35,220.94) .. (77.93,221.06) -- (69.83,192.18) -- cycle ; \draw  [dash pattern={on 4.5pt off 4.5pt}] (62.17,163.17) .. controls (77.4,159.13) and (93.34,167.63) .. (98.33,182.8) .. controls (103.5,198.54) and (94.94,215.5) .. (79.2,220.67) .. controls (78.78,220.81) and (78.35,220.94) .. (77.93,221.06) ;  
%Shape: Arc [id:dp15413881030008514] 
\draw  [draw opacity=0][dash pattern={on 4.5pt off 4.5pt}] (330.09,112.09) .. controls (315.95,119.07) and (298.65,113.89) .. (290.75,100.01) .. controls (282.56,85.61) and (287.59,67.29) .. (301.99,59.1) .. controls (302.38,58.88) and (302.77,58.67) .. (303.17,58.47) -- (316.83,85.18) -- cycle ; \draw  [dash pattern={on 4.5pt off 4.5pt}] (330.09,112.09) .. controls (315.95,119.07) and (298.65,113.89) .. (290.75,100.01) .. controls (282.56,85.61) and (287.59,67.29) .. (301.99,59.1) .. controls (302.38,58.88) and (302.77,58.67) .. (303.17,58.47) ;  
%Shape: Circle [id:dp41412502863376477] 
\draw  [fill={rgb, 255:red, 0; green, 0; blue, 0 }  ,fill opacity=1 ] (141.48,25.52) .. controls (141.48,24.23) and (142.53,23.18) .. (143.83,23.18) .. controls (145.12,23.18) and (146.18,24.23) .. (146.18,25.52) .. controls (146.18,26.82) and (145.12,27.87) .. (143.83,27.87) .. controls (142.53,27.87) and (141.48,26.82) .. (141.48,25.52) -- cycle ;

% Text Node
\draw (117,65) node [anchor=north west][inner sep=0.75pt]   [align=left] {$\displaystyle u >0$};
% Text Node
\draw (179,246) node [anchor=north west][inner sep=0.75pt]   [align=left] {$\displaystyle u< 0$};
% Text Node
\draw (261,177) node [anchor=north west][inner sep=0.75pt]   [align=left] {$\displaystyle u=0$};
% Text Node
\draw (50,196) node [anchor=north west][inner sep=0.75pt]   [align=left] {$\displaystyle x_{2}$};
% Text Node
\draw (320,57) node [anchor=north west][inner sep=0.75pt]   [align=left] {$\displaystyle x_{1}$};
% Text Node
\draw (324.83,209.18) node [anchor=north west][inner sep=0.75pt]   [align=left] {$\displaystyle x_{3}$};
% Text Node
\draw (291,257) node [anchor=north west][inner sep=0.75pt]   [align=left] {$\displaystyle \Omega$};
% Text Node
\draw (122,5) node [anchor=north west][inner sep=0.75pt]   [align=left] {$\displaystyle x_{0}$};

\end{tikzpicture}

\caption{Boundary regularity regimes.}\label{fig_intro}

\end{figure}

\bigskip

Our argument relies on approximation methods and imports regularity information from a homogeneous, fully nonlinear PDE back to \eqref{eq_main}. Our main ingredient is a pair of viscosity inequalities. Indeed, let $u\in C(\Omega)$ be an $L^p$-viscosity solution of
\begin{eqnarray}\label{eq_viscmin}
	\begin{split}
    		&\min\left(F_1(D^2u,Du,x),F_2(D^2u,Du,x)\right)\leq \left|f\right|& \hspace{.2in}\mbox{in}\hspace{.2in} \Omega,\\
		&\max\left(F_1(D^2u,Du,x),F_2(D^2u,Du,x)\right)\geq -\left|f\right|& \hspace{.2in}\mbox{in}\hspace{.2in} \Omega,
	\end{split}
\end{eqnarray}
satisfying 
\begin{equation}\label{eq_bdry}
	u=g\hspace{.2in}\mbox{on}\hspace{.2in}\partial\Omega.
\end{equation}

We suppose there exists a uniformly elliptic operator $F$ satisfying an appropriate closeness regime with respect to $F_1$ and $F_2$. Also, we require the source term $f$ to be in $L^p(B_1^+)$, for some $p>d$. Under these conditions, and geometric hypotheses on the domain $\Omega$, we establish $C^{1,{\rm Log-Lip}}$-boundary regularity estimates for the solutions to \eqref{eq_viscmin}-\eqref{eq_bdry}. This is the content of our first main result, stated with $\Omega=B_1^+$ for simplicity.

\begin{theorem}[$C^{1,\text{Log-Lip}}$-boundary regularity estimates]\label{Theo_Log-Lip}
Let  $u \in {C}(B_1^+)$ be a normalized viscosity solution of \eqref{eq_viscmin}-\eqref{eq_bdry}. Suppose Assumptions A\ref{assump_domain}, A\ref{assump_Felliptic}, and A\ref{assump_F-kappau}, to be detailed further, are in force. Suppose further $f\in L^p(B_1^+)$, for some $p>d$. Then $u\in C^{1,{\rm Log-Lip}}\left(B_1^0\right)$. That is, there exists a constant $C>0$ such that, for every $x_0 \in B_{1/2}^0$, we have 
	\begin{equation*}
		\sup_{x\in B_r^+(x_0)}\left|u(x)-u(x_0)-Du(x_0)(x-x_0)\right| \leq C\left|x-x_0\right|^2\ln\frac{1}{\left|x-x_0\right|}.
	\end{equation*}
Moreover, $C=C\left(d,\lambda,\Lambda,\partial\Omega, \left\|u\right\|_{L^\infty(B_1^+)},\left\|f\right\|_{L^p(B_1^+)}\right)$.
\end{theorem}

To improve the findings in the previous theorem, one imposes further conditions on the limiting profile $F$. Namely, suppose one requires $F$ to be differentiable with respect to the matrix entry. In that case, it is possible to prove that solutions to \eqref{eq_viscmin}-\eqref{eq_bdry} are in the Sobolev space $W^{2,p}$ in a neighbourhood of $\partial\Omega$. This is the content of our second main result.

\begin{theorem}[Sobolev boundary regularity]\label{thm_Sobolev}
Let $u\in C(B_1^+)$ be a viscosity solution to \eqref{eq_viscmin}-\eqref{eq_bdry}. Suppose Assumptions A\ref{assump_domain}, A\ref{assump_Felliptic}, A\ref{assump_F-kappa}, and A\ref{assump_FC1}, to be detailed further, are in force. Suppose further $f\in L^d(B_1^+)$. Then $u\in W^{2,d}(B_{1/2}^+)$ and there exists a universal constant $C>0$ such that 
\[
	\left\|u\right\|_{W^{2,d}(B_{1/2}^+)}\leq C\left(\left\|u\right\|_{L^\infty(B_1^+)}+\left\|f\right\|_{L^d(B_1^+)}\right).
\]
\end{theorem}

We remark on the connection between \eqref{eq_main} and \eqref{eq_viscmin}. Let $u\in W^{2,d}(\Omega)$ be an $L^d$-strong solution to \eqref{eq_main}. It is clear that $u$ is an $L^d$-strong solution to \eqref{eq_viscmin} and, therefore, an $L^d$-viscosity solution to the latter. Therefore, Theorems \ref{Theo_Log-Lip} and \ref{thm_Sobolev} yield information on strong solutions to \eqref{eq_main}. If $u\in C(\Omega)$ is an $L^p$-viscosity solution to \eqref{eq_main}, then it solves \eqref{eq_viscmin} in the $L^p$-viscosity sense in the subset $\left\lbrace u\neq 0\right\rbrace\cap \Omega$.

\begin{remark}[Propagation of boundary regularity]\normalfont
Regarding optimal regularity, one cannot expect the solutions to \eqref{eq_main} to have a H\"older-continuous Hessian. We produce an example where $F_1$ coincides with $F_2$ on the boundary, with both operators being convex. Moreover, $F_i(\cdot,x)\in C^\beta(B_1^+)$ for $i=1,2$. Still, the $C^{2,\alpha}$-regularity of the boundary data does not propagate to the domain's interior. Indeed, fix $d=2$ and consider the operator $F_i:S(d)\times B_1^+\to\mathbb{R}$, defined by
\[
	F_i(M,x):=\left(1+\left(-\frac{x_2}{2}\right)^i\right)^2{\rm Tr}(M),
\]
for $i=1,2$. Prescribe the transmission problem
\begin{equation}\label{eq_example}
	\begin{split}
		F_1(D^2u,x)=1&\hspace{.2in}\mbox{in}\hspace{.2in}B_1^+\cap\left\lbrace x_1<0\right\rbrace\\
		F_2(D^2u,x)=1&\hspace{.2in}\mbox{in}\hspace{.2in}B_1^+\cap\left\lbrace x_1>0\right\rbrace,
	\end{split}	
\end{equation}
with $u=0$ on $B_1^0$ and $u=g$ on $\partial B_1^+\setminus B_1^0$. Notice that in $B_1^+\cap\left\lbrace x_1<0\right\rbrace$ a solution to \eqref{eq_example} satisfies
\[
	\Delta u =\left(\frac{2}{2-x_2}\right)^2,
\]
whereas in $B_1^+\cap\left\lbrace x_1>0\right\rbrace$ it satisfies
\[
	\Delta u =\left(\frac{4}{4+x_2^2}\right)^2.
\]
The Laplacian of $u$ is discontinuous across $\left\lbrace x_1=0\right\rbrace$, leading to a discontinuous Hessian across this surface.
\end{remark}

The remainder of the paper is organised as follows. Our main assumptions are the subject of Section \ref{subsec_ma}, whereas Section \ref{subsec_auxiliaryresults} gathers preliminary notions and results used in the paper. We discuss the scaling properties of \eqref{eq_viscmin} in Section \ref{subsec_scaling}. Section \ref{sec_appmet} resorts to approximation methods to connect solutions to \eqref{eq_viscmin} with a limiting problem of the form $F=0$. We detail the proof of Theorem \ref{Theo_Log-Lip} in Section \ref{sec_proofthm1}. The final section presents the proof of Theorem \ref{thm_Sobolev}.

\section{Preliminaries}\label{sec_prelim}

In what follows, we detail our main assumptions and gather preliminary notions and results used in the paper. 

\subsection{Main Assumptions}\label{subsec_ma}

Our first assumption concerns the regularity of the domain $\Omega$.

\begin{Assumption}[Geometry of the domain]\label{assump_domain}
We suppose $\Omega\subset\mathbb{R}^d$ is a bounded domain. In addition, its boundary is locally the graph of a $C^{2,\beta}$-regular function for some fixed $\beta\in(0,1)$. We also suppose that $0 \in \partial \Omega$.
\end{Assumption}

We denote with $B_r(x)$ the ball of center $x\in\mathbb{R}^d$ and radius $r>0$, with $B_r(0)=B_r$. Given a point $x_0\in\partial\Omega$, we denote with $\Omega^+_{x_0}(r)$ the intersection
\[
	\Omega_{x_0}^+(r):=\Omega\cap B_r(x_0).
\]
Since we are interested in the boundary regularity of the solutions to \eqref{eq_viscmin}-\eqref{eq_bdry}, we also define $\Omega_{x_0}^0(r)$ as
\[
	\Omega_{x_0}^0(r):=\partial\Omega\cap B_r(x_0).
\]

The operators $F_1$ and $F_2$ satisfy a structural condition. Before detailing it, we recall the definition of the extremal Pucci operators.

\begin{definition}[Extremal Pucci operators]\label{def_pucci}
Fix $0 < \lambda \leq \Lambda$. We define the extremal Pucci operators $\mathcal{M}^-_{\lambda,\Lambda},\mathcal{M}^+_{\lambda,\Lambda}: S(d)\to\mathbb{R}$ as
\[
	\mathcal{P}^+_{\lambda,\Lambda}(M) = \Lambda\sum_{e_i > 0} e_i + \lambda \sum_{e_i < 0} e_i
\]
and
\[
	\mathcal{P}^-_{\lambda,\Lambda}(M) = \lambda\sum_{e_i > 0} e_i + \Lambda \sum_{e_i < 0} e_i,
\]
where $e_i$ are the eigenvalues of $M$.
\end{definition}

\begin{Assumption}[Structural condition]\label{assump_Felliptic}
Fix $0<\lambda\leq\Lambda$ and $K>0$. For $i=1,2$, we suppose the operator $F_i:S(d)\times\mathbb{R}^d\times\Omega\to\mathbb{R}$ satisfies
	\[
	\mathcal{M}^-_{\lambda,\Lambda}(M-N) -K\left|p-q\right| \leq F_i(M,p,x)-F_i(N,q,x)\leq \mathcal{M}^+_{\lambda,\Lambda}(M-N) +K\left|p-q\right|,
	\]
for every $M,N\in S(d)$, $p,q\in\mathbb{R}^d$ and $x\in\Omega$.
\end{Assumption}

The former assumption builds upon A\ref{assump_domain} allowing us to work under a flatness hypothesis on $\partial\Omega$ while remaining in the same class of operators; see \cite[Proposition 2.1]{SilSir_2014}. 

\begin{Assumption}[Closeness condition]\label{assump_F-kappa}
	For $i=1,2$, we suppose there exists a uniformly elliptic operator $F:S(d)\times\mathbb{R}^d\to\mathbb{R}$, and a constant $0<\kappa\ll1$ such that
\[	
	\left|F_i(M,p,x)-F(M,p)\right| \leq \kappa (1+|p|+|M|),
\]
for every $x\in \Omega$, $p\in\mathbb{R}^d$ and $M\in S(d)$.
\end{Assumption}

\begin{Assumption}[Uniform closeness condition]\label{assump_F-kappau}
	For $i=1,2$, we suppose there exists a uniformly elliptic operator $F:S(d)\times\mathbb{R}^d\to\mathbb{R}$, and a constant $0<\tau\ll1$ such that
\[	
	\left|F_i(M,p,x)-F(M,p)\right| \leq \tau (1+|p|),
\]
for every $x\in \Omega$, $p\in\mathbb{R}^d$ and $M\in S(d)$.
\end{Assumption}

Although the universal constants $0<\kappa,\tau\ll 1$ in Assumptions A\ref{assump_F-kappa} and A\ref{assump_F-kappau} are fixed, they are determined further in the paper. When examining Sobolev regularity, we require the limiting profile $F=F(M,p)$ to be differentiable with respect to $M$.

\begin{Assumption}[Differentiability of the limiting profile]\label{assump_FC1}
	We suppose the operator $F:S(d)\times\mathbb{R}^d\to\mathbb{R}$ is of class $C^1$ concerning its first entry.
\end{Assumption}

Under Assumption A\ref{assump_FC1}, we denote with $D_MF$ the derivative of $F$ with respect to $M$. The modulus of continuity for $D_MF$ is denoted with $\omega_F$. In what follows, we gather preliminary results used in the paper.

\subsection{Preliminary notions and auxiliary results}\label{subsec_auxiliaryresults}

We start with the definition of $L^p$-viscosity solutions for an equation of the form
\begin{equation}\label{eq_lorvao}
	G(D^2u,Du,u,x)=f\hspace{.2in}\mbox{in}\hspace{.2in}\Omega.
\end{equation}

\begin{definition}[Viscosity solution]\label{def_visclp}
Let $p>d/2$. We say that $u\in C(\Omega)$ is an $L^p$-viscosity sub-solution to \eqref{eq_lorvao} if, whenever $\phi\in W^{2,p}_{loc}(\Omega)$ is such that $u-\phi$ has a local minimum at $x_0\in\Omega$, we have
\[
	{\rm ess}\limsup_{x\to x_0}\left(G(D^2\phi(x),D\phi(x),u(x),x)-f(x)\right)\geq 0.
\]
We say that $u\in C(\Omega)$ is an $L^p$-viscosity super-solution to \eqref{eq_lorvao} if, whenever $\phi\in W^{2,p}_{loc}(\Omega)$ is such that $u-\phi$ has a local maximum at $x_0\in\Omega$, we have
\[
	{\rm ess}\liminf_{x\to x_0}\left(G(D^2\phi(x),D\phi(x),u(x),x)-f(x)\right)\leq 0.
\]
If $u\in C(\Omega)$ is an $L^p$-viscosity sub-solution and an $L^p$-viscosity super-solution to \eqref{eq_lorvao}, we say it is an $L^p$-viscosity solution to \eqref{eq_lorvao}.
\end{definition}

When testing against smooth functions $\phi\in C^2(\Omega)$, Definition \ref{def_visclp} yields the usual notion of $C$-viscosity solutions \cite[Section 2]{CIL}. We say that $u\in C(\Omega)$ is a \emph{normalized} viscosity solution if $\left\|u\right\|_{L^\infty(\Omega)}\leq 1$. Next, we recall a version of the Aleksandrov-Bakelman-Pucci maximum principle; see \cite[Theorem 1.1]{KoikeSwiech2022}.

\begin{proposition}[Aleksandrov-Bakelman-Pucci estimate]\label{prop_abp}
Fix $q>d$. Let $\gamma\in L^q(\Omega)$ and $f\in L^d(\Omega)$. Suppose $u\in C(\Omega)$ is a viscosity solution to 
\[
	\mathcal{M}^+_{\lambda,\Lambda}(D^2u)+\gamma(x)|Du|\geq -f\hspace{.2in}\mbox{in}\hspace{.2in}\Omega
\]
\[
	\left[\mbox{resp.}\hspace{.1in}\mathcal{M}^-_{\lambda,\Lambda}(D^2u)-\gamma(x)|Du|\leq f\hspace{.2in}\mbox{in}\hspace{.2in}\Omega\right].
\]
Then there exists a positive constant $C=C\left((d,\lambda,\Lambda,\left\|\gamma\right\|_{L^d(\Omega)}\right)$ such that 
\[
	\max_{\overline\Omega}u\leq \max_{\partial\Omega}u+C\,{\rm diam}(\Omega)\left\|f\right\|_{L^d(\Gamma^+(u))}
\]
\[
	\left[\mbox{resp.}\hspace{.1in}\max_{\overline\Omega}(-u)\leq \max_{\partial\Omega}(-u)+C\,{\rm diam}(\Omega)\left\|f\right\|_{L^d(\Gamma^+(-u))}\right].
\]
\end{proposition}

Under the structural condition in Assumption A\ref{assump_Felliptic}, one notices that a solution $u\in C(\Omega)$ to the first equation in \eqref{eq_viscmin} satisfies
\[
	\mathcal{M}^-_{\lambda,\Lambda}(D^2u)-K|Du|\leq f\hspace{.2in}\mbox{in}\hspace{.2in}\Omega.
\]
Also, a solution $v\in C(\Omega)$ to the second equation in \eqref{eq_viscmin} satisfies
\[
	\mathcal{M}^+_{\lambda,\Lambda}(D^2u)+K|Du|\geq f\hspace{.2in}\mbox{in}\hspace{.2in}\Omega.
\]
Hence, the upper bounds in Proposition \ref{prop_abp} are available for the viscosity solutions to \eqref{eq_viscmin}. We proceed by recalling boundary-regularity results for viscosity solutions to fully nonlinear equations. Our strategy is to import information from those results into the context of the free transmission problems.

The next results account for the boundary regularity of viscosity solutions to
\medskip
\begin{equation}\label{eq_limit}
	\begin{cases}
		F(D^2u,Du)=0&\hspace{.2in}\mbox{in}\hspace{.2in}\Omega\\
		u=0&\hspace{.2in}\mbox{on}\hspace{.2in}\partial\Omega.
	\end{cases}
\end{equation}

\medskip

\begin{proposition}[Boundary regularity]\label{prop_boundary1}
Let $u\in C(\Omega)$ be a viscosity solution to \eqref{eq_limit}. Suppose Assumptions A\ref{assump_domain} and A\ref{assump_Felliptic} are in force. Then there exists a H\"older continuous function $H : B^0_{1/2} \to \mathbb{R}^{d \times d}$, and constants $\alpha=\alpha(d,\lambda,\Lambda)\in(0,1)$ and $C=C(d,\lambda,\Lambda,K)>0$ such that 
\[
	F(H, Du) = 0\hspace{.3in}\mbox{on}\hspace{.3in}B^0_{1/2},
\]
and
\[
	\left|u(x) - u(x_0)-Du(x_0) \cdot (x - x_0) - \frac{H(x_0) (x - x_0) \cdot (x - x_0)}{2} \right| \leq C |x - x_0 |^{2 + \alpha},
\]
for every $x \in B^+_1$, and $ x_0 \in B^0_{1/2}$.
Moreover,
\[
	\left\|H\right\|_{C^\alpha \left(B^0_{1/2}\right)} \leq C \left\|u\right\|_{L^\infty (B^+_1)}.
\]
\end{proposition}

In the former proposition, $H$ represents the Hessian of $u$. For a proof of Proposition \ref{prop_boundary1}, we refer the reader to \cite[Theorem 1.2]{SilSir_2014}. The next proposition extends the Hessian regularity obtained above to a $\delta$-neighbourhood of $\partial\Omega$. To be more precise, let $0<\delta\ll1$; we define $\Omega_\delta\subset\Omega$ as
\[
	\Omega_\delta:=\left\lbrace x\in \Omega\,:\, \dist(x,\partial \Omega)< \delta\right\rbrace.
\]

\begin{proposition}[Interior regularity near the boundary]\label{prop_C2alpha_SS}
Let $u\in C(\Omega)$ be an $L^p$-viscosity solution to \eqref{eq_limit}. Suppose Assumptions A\ref{assump_domain}, A\ref{assump_Felliptic}, A\ref{assump_F-kappa}, and A\ref{assump_FC1} are in force. Then there exist universal constants $\alpha\in(0,1)$, and $C>0$, such that $u\in C^{2,\alpha}(\Omega_\delta)$ and
\[
	\left\|u\right\|_{C^{2,\alpha}(\Omega_\delta)}\leq C\|g\|_{L^\infty(\partial \Omega)}.
\]
Here, $0<\delta\ll1$ depends on universal constants, the boundary $\partial \Omega$, and $\omega_F$ on a ball $B_R\subset S(d)$, with universal radius $R>0$ depending also on $\partial\Omega$.
\end{proposition}

For the proof of Proposition \ref{prop_C2alpha_SS}, we refer the reader to \cite[Theorem 1.3]{SilSir_2014}. We conclude this section by collecting definitions and notation used in the study of Sobolev boundary regularity (see \cite[Chapter 7]{Caffarelli_Cabre1995}; see also \cite[Definition 4]{PimSanTei2022}).

A paraboloid of opening $M>0$ is a function of the form
\[
	P_M(x):=\ell(x)\pm\frac{1}{2}M\left|x\right|^2,
\]
where $\ell(\cdot)$ is an affine function. The choice of sign determines whether $P$ is concave or convex.
\begin{definition}\label{def_goodsets}
Let $U \subset \Omega$ be an open subset, and take $0<\rho<{\rm diam}(O)/\pi$. For $M > 0$, define
\[
	\underline{G}_M(u, U) = \underline{G}_M(U)
\]
as the set of all points $z \in U$ for which there exists a concave paraboloid $P_M$ satisfying
\begin{enumerate}
	\item $u(z) = P_M(z)$;
	\item $u(x) > P_M(x)$ for all $x \in B_{\rho}(z)$, with $x\neq z$.
\end{enumerate}
Also, set
\[
	\overline{G}_M(u, U) = \overline{G}_M(U)
\]
as the set of all points $z \in U$ for which there exists a convex paraboloid $P_M$ satisfying
\begin{enumerate}
	\item $u(z) = P_M(z)$;
	\item $u(x) < P_M(x)$ for all $x \in B_{\rho}(z)$, with $x\neq z$.
\end{enumerate}
Finally 
\[
G_M(U) = \underline{G}_M(U)\cap \overline{G}_M(U).
\]
\end{definition}
We also define
\[
	\underline{A}_M(U):=U\setminus \underline{G}_M(U)
\]
\[	
	\overline{A}_M(U):=U\setminus \overline{G}_M(U)
\]
and
\[
	A_M(U):=U\setminus G_M(U).
\]

Using this notation, the \emph{opening function} associated with $u\in C(\Omega)$, with respect to the subset $B\subset\Omega$, is denoted with $\Theta(x,B):B\to[0,+\infty]$ and given by
\begin{equation}\label{eq_petrenko}
	\Theta(u,B)(x):=\inf\left\lbrace M\,|\,x\in G_M(B)\right\rbrace\in [0,\infty],\quad x\in B.
\end{equation}
The integrability of $\Theta$ is equivalent to the regularity of $u$ in Sobolev spaces. Hence, we recall the distribution function of $\Theta$ in a set $B$, denoted with $\mu_\Theta$ and given by
\[
	\mu_\Theta (t):=\left|\left\lbrace x\in B \,|\, \Theta(x)>t\right\rbrace\right|.
\]
For completeness, we recall a lemma relating the integrability of $\Theta$ with the summability of its distribution function $\mu_\Theta$.

\begin{lemma}\label{lemma_distrLp}
Let $\Theta$ be defined as in \eqref{eq_petrenko} and set $B:=\Omega$. Fix constants $\eta>0$ and $M>1$. For $0<p<\infty$, we have $\Theta\in L^p(\Omega)$ if, and only if, 
\[
	S:=\sum_{k\geq 1} M^{pk} \mu_\Theta(\eta M^k)<\infty.
\]
In the case $S$ is finite, we also have
\[
	C^{-1}S\leq \left\|\Theta\right\|^p_{L^p(\Omega)}\leq C\left(\left|\Omega\right|+S\right),
\]
for some constant $C>0$, depending only on $\eta$, $M$ and $p$. 
\end{lemma}
As a consequence, the $p$-integrability of $\Theta$ amounts to the analysis of the summability of $M^{pk} \left|A_{M^k}(B)\right|$. Lastly, recall the definition of \emph{maximal function} associated with an element $f\in L^1_{loc}(\Rr^d)$; it is defined by
\[
	M(f)(x):=\sup_{r>0}\frac{1}{|Q_r|}\int_{Q_r(x)}|f(y)|\,{\rm d}y.
\]
We use elementary properties of the maximal function. Namely, that it is an operator of weak type $(1,1)$ and of strong type $(p,p)$, for $1<p\leq \infty$. That is, there exists $C>0$ such that 
\[
	\left|\left\lbrace x\in \Rr^d\,|\,M(f)\geq t\right\rbrace\right|\leq \frac{C}{t}\left\|f\right\|_{L^1(\Rr^d)},
\]
for all $t>0$ and
\[
	\left\|M(f)\right\|_{L^p(\Rr^d)}\leq C\left\|f\right\|_{L^p(\Rr^d)},
\]
for every $1<p\leq \infty$. In what follows, we discuss scaling properties of \eqref{eq_viscmin}. 

\subsection{Scaling properties}\label{subsec_scaling}

Let $u\in C(\Omega)$ be a viscosity solution to \eqref{eq_viscmin} and define $\overline u\in C(\Omega)$ as
\[
	\overline u(x):=\frac{u(rx)}{K},
\]
for some $0<r\leq 1$ and $K>0$. A simple computation ensures that $\overline u$ solves 
\[
	\min\left(\overline F_1(D^2\overline u, D\overline u,x),\overline F_2(D^2\overline u, D\overline u,x)\right)\leq \left\|\overline f\right\|_{L^\infty(\Omega)}
\]
and
\[
	\max\left(\overline F_1(D^2\overline u, D\overline u,x),\overline F_2(D^2\overline u, D\overline u,x)\right)\geq -\left\|\overline f\right\|_{L^\infty(\Omega)}
\]
where
\[
	\overline F_i(M,p,x):=\frac{r^2}{K}F_i\left(\frac{K}{r^2}M,\frac{K}{r}p,rx\right),\hspace{.4in}\mbox{for}\;i=1,2,
\]
and
\[
	\overline f(x)=\frac{r^2}{K}f(rx).
\]
Clearly, $\overline F_i$ satisfies the structural condition in Assumption A\ref{assump_Felliptic}. By taking $r=1$ and $K:=1+\left\|u\right\|_{L^\infty(\Omega)}+\left\|f\right\|_{L^\infty(\Omega)}$, we suppose $u$ to be a normalized viscosity solution to \eqref{eq_viscmin} and $\left\|f\right\|_{L^\infty(\Omega)}$ to be as small as required. 

Finally, standard covering arguments build upon Assumption A\ref{assump_domain}, allowing us to consider $\Omega=B_1^+$ and to study the boundary regularity along the flat boundary $B_1^0$. For simplicity, we set $\Omega=B_1^+$ in the sequel. When dealing with estimates in Sobolev spaces, we consider $\Omega=B_{14\sqrt{d}}^+$ to fine-tune with the usual notation \cite{Caffarelli_Cabre1995,WinterNiki}. The next section concerns approximation methods.

\section{Approximation methods}\label{sec_appmet}

In this section, we establish an approximation lemma under Assumptions A\ref{assump_F-kappa} and A\ref{assump_FC1}. Namely, requiring the limiting operator $F$ to be locally uniformly close to $F_i$ and to be differentiable with respect to the Hessian of solutions. 

This result ensures that solutions to \eqref{eq_viscmin} can be arbitrarily approximated by $C^{2,\alpha}$-regular functions \emph{near} the boundary $\partial\Omega$, with estimates. More precisely, it guarantees the existence of a function $h\in C^{2,\alpha}(\Omega_\delta)$ approximating $u$ in the $L^\infty$-norm, for some small $0<\delta\ll1$, depending on the dimension, $\lambda$, $\Lambda$ and $\partial\Omega$. It is instrumental in examining Sobolev boundary regularity for the solutions to \eqref{eq_viscmin}. This is the content of Proposition \ref{prop_approx_w2p}.

Afterwards, we remove Assumption \ref{assump_FC1} and work under Assumptions A\ref{assump_Felliptic} and A\ref{assump_F-kappau}. That is, we suppose $F_1$ and $F_2$ are close to a merely $(\lambda,\Lambda)$-elliptic operator $F$. However, such closeness is uniform with respect to the matrix entry. In this setting, we obtain an approximation lemma involving functions satisfying a $C^{2,\alpha}$-regularity estimate \emph{along the boundary} $\partial\Omega$. The approximation regimes in Proposition \ref{prop_approx_w2p} and Corollary \ref{cor_approx_2} are used in the proofs of Theorems \ref{thm_Sobolev} and \ref{Theo_Log-Lip}, respectively.

\begin{remark}\label{rem_scalingfromhell}\normalfont
As mentioned, we use standard flattening, covering and rescaling arguments to replace $\Omega$ with $B_{14\sqrt{d}}^+$. Moreover, instead of obtaining results in $\Omega_\delta$ we localize the arguments in the set $B_{12\sqrt{d}}^0\times(0,\delta)$. Because the boundary data is $C^{2,\alpha}$-regular, we can suppose $g$ is identically zero. Then we define, for $K=\delta/(14\sqrt{d})$, the rescaling
\[
	\bar u(x',x_d)= u(x',K x_d).
\]
Notice $\overline u$ solves the same inequalities as $u$ with $F_i$, $F$ and $f$ replaced with $\tilde F_i(M)=K^2 F_i(K^{-2}M)$, $\tilde F(M)=K^2 F(K^{-2}M)$ and $\tilde f(x',x_d)=K^2 f(x',K x_d)$. Note all the relevant assumptions still hold, since $\delta$ is a universal constant. Thus we transform the set $B_{14\sqrt{d}}^0\times(0,\delta)$ into $B_{14\sqrt{d}}^0\times(0,14\sqrt{d})$. From now on, to ease notation, we set $\Omega:= B_{14\sqrt{d}}'\times(0,14\sqrt{d})$.
\end{remark}

\begin{proposition}[Approximation Lemma I]\label{prop_approx_w2p}
Let $u\in C(B_1)$ be a normalized viscosity solution to \eqref{eq_viscmin}-\eqref{eq_bdry}. Suppose Assumptions A\ref{assump_domain}-A\ref{assump_F-kappa} and A\ref{assump_FC1} are in force.
Then there exists a function $h\in C^{2,\alpha}(B_{12\sqrt{d}}^+)$ and $\varphi\in C(B_{12\sqrt{d}}^+)$ such that $u-h\in S^*(\varphi)$ in $B_{12\sqrt{d}}^+$,
\[
	\left\|h\right\|_{C^{2,\alpha}\left(B_{12\sqrt{d}}^+\right)}\leq C\left\|u\right\|_{L^\infty(B_{14\sqrt{d}}'\times(0,14\sqrt{d}))},
\]
and
\begin{align*}
\|u-h\|_{L^\infty(B_{12\sqrt{d}}^+)}+\|\varphi\|_{L^d(B_{12\sqrt{d}}^+)}\leq C\left(\kappa^\gamma+\|f\|_{L^d(B_{14\sqrt{d}}'\times(0,14\sqrt{d}))}\right),
\end{align*}
for some universal exponents $\alpha, \gamma\in(0,1)$.
\end{proposition}
\begin{proof}
Let $h\in C(B_{13\sqrt{d}}^+)$ be the viscosity solution to the problem
\begin{equation*}
\begin{cases}
 F(D^2h,Dh)=0 &\mbox{ in } B_{13\sqrt{d}}^+\\
h=u &\mbox{ on } \partial B_{13\sqrt{d}}^+.
\end{cases}
\end{equation*}
It follows from Proposition \ref{prop_C2alpha_SS}, rescaled accordingly, that $h\in C^{2,\alpha}(B_{12\sqrt{d}}^+)$, for some universal $\alpha\in (0,1)$. Since $u\in S^*(f)$, we have $u\in C^{0,\beta}(B_{13\sqrt{d}}^+)$ with
\[
	\left\|u\right\|_{C^{0,\beta}\left(B_{13\sqrt{d}}^+\right)}\leq C\left(\left\|u\right\|_{L^\infty}+\left\|f\right\|_{L^d\left(B_{13\sqrt{d}}^+\right)}\right).
\]
Therefore we also have, by \cite[Proposition 4.13]{Caffarelli_Cabre1995},
\begin{equation*}
	\left\|h\right\|_{C^{0,\alpha}\left(B_{13\sqrt{d}}^+\right)}\leq C\left(1+\left\|u\right\|_{C^{0,\beta}\left(B_{13\sqrt{d}}^+\right)}\right)\leq C\left(1+\left\|f\right\|_{L^d\left(B_{14\sqrt{d}}^+\right)}\right),
\end{equation*}
where $0<\alpha<\beta$ are universal constants. Let $0<\theta\ll1$; since $u-h=0$ on $\partial B_{13\sqrt{d}}^+$, for every $x\in \partial B_{13\sqrt{d}-\theta}^+$ one gets
\[
	(u-h)(x)\leq \theta^\alpha\left\|u-h\right\|_{C^{0,\alpha}\left(B_{13\sqrt{d}}^+\right)}.
\]
Hence
\begin{equation}\label{eq:18}
	\left\|u-h\right\|_{L^\infty\left(\partial B_{13\sqrt{d}-\theta}^+\right)}\leq C\theta^\alpha\left(1+\left\|f\right\|_{L^d\left(B_{14\sqrt{d}}^+\right)}\right).
\end{equation}

For fixed $x_0\in B_{13\sqrt{d}-\theta}^+$, there are two possibilities. Either $B_{\theta/2}(x_0)\subset B_{13\sqrt{d}}^+$ or there exists a point $z_0\in B_{13\sqrt{d}-\theta}'$ such that $ x_0\in B_{\theta/2}^+(z_0)$. In any case, Proposition \ref{prop_C2alpha_SS} applied to $h-h(x_0)$ in $B_{\theta/2}(x_0)$ yields
\[
	\theta^2\|D^2h(x)\|\leq C\theta^\alpha\left(1+\|f\|_{L^d\left(B_{14\sqrt{d}}^+\right)}\right).
\]

Using the closeness regime in Assumption A\ref{assump_F-kappa}, one gets
\begin{equation}\label{eq_newregime}
	\left|F_i(M,p,x)- F(M,p)\right|\leq \kappa\left(1+\left|p\right|+\left|M\right|\right);
\end{equation}
the former inequality combined with and $ F(D^2h,Dh)=0$, produces
\begin{equation}\label{eq:19}
	\left|F_i(D^2h(x),Dh(x),x)\right|\leq C\theta^{\alpha-2}\kappa\left(1+\left\|f\right\|_{L^d\left(B_{14\sqrt{d}}^+\right)}\right).
\end{equation}
Therefore, we conclude the inequalities
\begin{equation*}
	\begin{cases}
		\min\left(F_1(D^2h,Dh,x),F_2(D^2h,Dh,x)\right)\leq C\theta^{\alpha-2}\kappa\left(1+\left\|f\right\|_{L^d\left(B_{14\sqrt{d}}^+\right)}\right)\\
		\max\left(F_1(D^2h,Dh,x),F_2(D^2h,Dh,x)\right)\geq -C\theta^{\alpha-2}\kappa\left(1+\left\|f\right\|_{L^d\left(B_{14\sqrt{d}}^+\right)}\right)
	\end{cases}
\end{equation*}
hold in the classical sense. Hence $u-h\in S^*(\varphi)$ in $B_{12\sqrt{d}}^+$, with
\[
	\left\|\varphi\right\|_{L^d\left(B_{12\sqrt{d}}^+\right)}\leq \left\|f\right\|_{L^d\left(B_{14\sqrt{d}}^+\right)}+C\theta^{\alpha-2}\kappa\left(1+\left\|f\right\|_{L^d\left(B_{14\sqrt{d}}^+\right)}\right).
\]
Combining this with \eqref{eq:18} we get
\[
	\begin{split}
		\|u-h\|_{L^\infty(B_{12\sqrt{d}}^+)}+\|\varphi\|_{L^d(B_{12\sqrt{d}}^+)}&\leq C\theta^\alpha\left(1+\left\|f\right\|_{L^d\left(B_{14\sqrt{d}}^+\right)}\right)\\
			&\quad+C\theta^{\alpha-2}\kappa\left(1+\|f\|_{L^d\left(B_{14\sqrt{d}}^+\right)}\right)\\
			&\quad+\left\|f\right\|_+\|f\|_{L^d\left(B_{14\sqrt{d}}^+\right)}.
	\end{split}
\]
By choosing $\theta^2<\kappa$ and setting $\gamma:=\alpha/2$ onde completes the proof.
\end{proof}

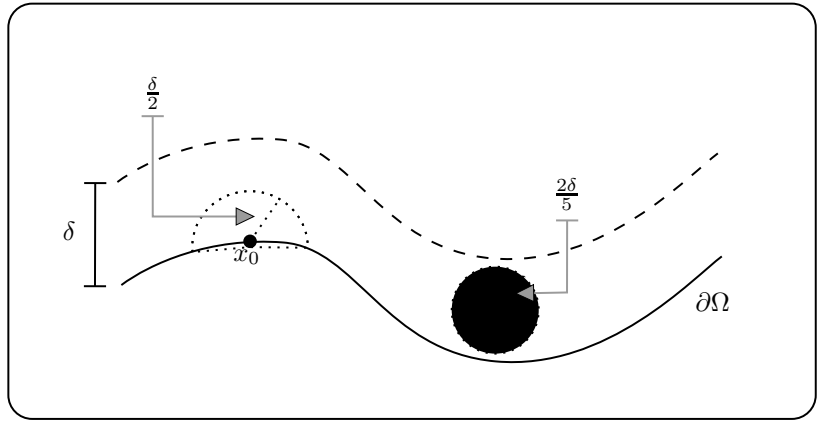
\begin{figure}[h!]
\centering

\begin{tikzpicture}[x=0.75pt,y=0.75pt,yscale=-1,xscale=1]

%Curve Lines [id:da2069242822854146] 
\draw  [dash pattern={on 4.5pt off 4.5pt}]  (100,101) .. controls (114.23,90.33) and (144.83,76.52) .. (181.83,79.52) .. controls (218.83,82.52) and (232.83,134.52) .. (288.83,139.52) .. controls (344.83,144.52) and (388.43,97.32) .. (402.83,86.52) ;
%Curve Lines [id:da0907392687700247] 
\draw    (102,153) .. controls (116.23,142.33) and (146.83,128.52) .. (183.83,131.52) .. controls (220.83,134.52) and (234.83,186.52) .. (290.83,191.52) .. controls (346.83,196.52) and (390.43,149.32) .. (404.83,138.52) ;
%Straight Lines [id:da3417678090290841] 
\draw    (88.83,101.52) -- (88.7,153.58) ;
\draw [shift={(88.7,153.58)}, rotate = 270.14] [color={rgb, 255:red, 0; green, 0; blue, 0 }  ][line width=0.75]    (0,5.59) -- (0,-5.59)   ;
\draw [shift={(88.83,101.52)}, rotate = 270.14] [color={rgb, 255:red, 0; green, 0; blue, 0 }  ][line width=0.75]    (0,5.59) -- (0,-5.59)   ;
%Shape: Circle [id:dp1002922331985232] 
\draw  [fill={rgb, 255:red, 0; green, 0; blue, 0 }  ,fill opacity=1 ] (164,130.91) .. controls (164,129.3) and (165.3,128) .. (166.91,128) .. controls (168.52,128) and (169.83,129.3) .. (169.83,130.91) .. controls (169.83,132.52) and (168.52,133.83) .. (166.91,133.83) .. controls (165.3,133.83) and (164,132.52) .. (164,130.91) -- cycle ;
%Shape: Arc [id:dp04567334755459651] 
\draw  [draw opacity=0][dash pattern={on 0.84pt off 2.51pt}] (137.9,136.35) .. controls (137.78,135.02) and (137.75,133.67) .. (137.82,132.3) .. controls (138.64,116.74) and (152.32,104.8) .. (168.39,105.64) .. controls (184.05,106.46) and (196.19,119.11) .. (196.05,134.16) -- (166.91,133.83) -- cycle ; \draw  [dash pattern={on 0.84pt off 2.51pt}] (137.9,136.35) .. controls (137.78,135.02) and (137.75,133.67) .. (137.82,132.3) .. controls (138.64,116.74) and (152.32,104.8) .. (168.39,105.64) .. controls (184.05,106.46) and (196.19,119.11) .. (196.05,134.16) ;  
%Straight Lines [id:da7086033399932701] 
\draw  [dash pattern={on 0.84pt off 2.51pt}]  (166.91,130.91) -- (181.77,109.65) ;
%Straight Lines [id:da7668407951295515] 
\draw [color={rgb, 255:red, 155; green, 155; blue, 155 }  ,draw opacity=1 ]   (117.81,67.75) -- (117.83,118.39) -- (165.83,118.39) ;
\draw [shift={(168.83,118.39)}, rotate = 180] [fill={rgb, 255:red, 155; green, 155; blue, 155 }  ,fill opacity=1 ][line width=0.08]  [draw opacity=0] (8.93,-4.29) -- (0,0) -- (8.93,4.29) -- cycle    ;
\draw [shift={(117.81,67.75)}, rotate = 269.98] [color={rgb, 255:red, 155; green, 155; blue, 155 }  ,draw opacity=1 ][line width=0.75]    (0,5.59) -- (0,-5.59)   ;
%Shape: Circle [id:dp20250703285573546] 
\draw  [fill={rgb, 255:red, 0; green, 0; blue, 0 }  ,fill opacity=1 ] (287.64,165.65) .. controls (287.64,164.04) and (288.95,162.74) .. (290.56,162.74) .. controls (292.17,162.74) and (293.47,164.04) .. (293.47,165.65) .. controls (293.47,167.26) and (292.17,168.57) .. (290.56,168.57) .. controls (288.95,168.57) and (287.64,167.26) .. (287.64,165.65) -- cycle ;
%Shape: Circle [id:dp20514146564002556] 
\draw  [fill={rgb, 255:red, 0; green, 0; blue, 0 }  ,fill opacity=0.15 ][dash pattern={on 0.84pt off 2.51pt}] (268.61,165.65) .. controls (268.61,153.53) and (278.44,143.71) .. (290.56,143.71) .. controls (302.68,143.71) and (312.5,153.53) .. (312.5,165.65) .. controls (312.5,177.77) and (302.68,187.6) .. (290.56,187.6) .. controls (278.44,187.6) and (268.61,177.77) .. (268.61,165.65) -- cycle ;
%Straight Lines [id:da16157825104986756] 
\draw  [dash pattern={on 0.84pt off 2.51pt}]  (290.56,165.65) -- (305.47,148.57) ;
%Straight Lines [id:da10157433153526474] 
\draw [color={rgb, 255:red, 155; green, 155; blue, 155 }  ,draw opacity=1 ]   (326.83,120.37) -- (326.77,156.73) -- (304.02,157.17) ;
\draw [shift={(301.02,157.23)}, rotate = 358.89] [fill={rgb, 255:red, 155; green, 155; blue, 155 }  ,fill opacity=1 ][line width=0.08]  [draw opacity=0] (8.93,-4.29) -- (0,0) -- (8.93,4.29) -- cycle    ;
\draw [shift={(326.83,120.37)}, rotate = 270.09] [color={rgb, 255:red, 155; green, 155; blue, 155 }  ,draw opacity=1 ][line width=0.75]    (0,5.59) -- (0,-5.59)   ;
%Rounded Rect [id:dp9126852317027754] 
\draw   (45.01,22.95) .. controls (45.01,16.32) and (50.38,10.95) .. (57.01,10.95) -- (440.21,10.95) .. controls (446.84,10.95) and (452.21,16.32) .. (452.21,22.95) -- (452.21,208.55) .. controls (452.21,215.18) and (446.84,220.55) .. (440.21,220.55) -- (57.01,220.55) .. controls (50.38,220.55) and (45.01,215.18) .. (45.01,208.55) -- cycle ;

% Text Node
\draw (71,119.4) node [anchor=north west][inner sep=0.75pt]    {$\delta $};
% Text Node
\draw (390,155.4) node [anchor=north west][inner sep=0.75pt]    {$\partial {\Omega}$};
% Text Node
\draw (107,41.4) node [anchor=north west][inner sep=4.5pt]    {$\frac{\delta}{2}$};
% Text Node
\draw (157,133.4) node [anchor=north west][inner sep=0.75pt]    {$x_{0}$};
% Text Node
\draw (272.75,164.15) node [anchor=north west][inner sep=4pt]    {$x_{1}$};
% Text Node
\draw (314,93.12) node [anchor=north west][inner sep=4pt]    {$\frac{2\delta}{5}$};

\end{tikzpicture}

\caption{{\small Proposition \ref{prop_approx_w2p} profits from the estimates in Proposition \ref{prop_C2alpha_SS}. Indeed, the approximating function $h$ solves a PDE with $C^{2,\alpha}$ estimates in a $\delta$-vicinity of $\partial\Omega$. Therefore, if we focus on  points $x_0\in\partial\Omega$ or in $x_1\in \Omega_\delta$, the function $h\in C^{2,\alpha}(\Omega_\delta)$ with estimates. This idea unlocks the proof of Proposition \ref{prop_C2alpha_SS}.}}\label{fig_c2aintbdry}

\end{figure}

By replacing the Assumption A\ref{assump_F-kappa} with A\ref{assump_F-kappau} in Proposition \ref{prop_approx_w2p}, the inequality in \eqref{eq_newregime} becomes
\[
	\left|F_i(M,p,x)- F(M,p)\right|\leq \tau\left(1+\left|p\right|\right),
\]
for every $M\in S(d)$, $p\in\mathbb{R}^d$, $x\in B_1^+$, and every $i=1,2$. Hence, a simplified version of the proof of Proposition \ref{prop_approx_w2p} produces the following approximation result.

\begin{corollary}[Approximation Lemma II]\label{cor_approx_2}
Let $u\in C(B_1^+)$ be a normalized viscosity solution to \eqref{eq_viscmin}-\eqref{eq_bdry}. Suppose Assumptions A\ref{assump_domain}, A\ref{assump_Felliptic}, and A\ref{assump_F-kappau} are in force.
Then for every $\varepsilon>0$ there exists $\tau>0$ such that, if A\ref{assump_F-kappau} holds for such a $\tau>0$, one finds a function $h\in C^{2,\alpha}(B_{1/2}^0)$ satisfying
\[
	\left\|h\right\|_{C^{2,\alpha}\left(B_{1/2}^+\right)}\leq C\left\|u\right\|_{L^\infty(B_1^+)}
\]
and
\begin{align*}
	\left\|u-h\right\|_{L^\infty(B_{1/2}^+)}\leq \varepsilon.
\end{align*}

\end{corollary}

In the next section, an application of Corollary \ref{cor_approx_2} unlocks the proof of Theorem \ref{Theo_Log-Lip}.

\section{Boundary regularity in $C^{1,{\rm Log-Lip}}$-spaces}\label{sec_proofthm1}

In the sequel, we detail the proof of Theorem \ref{Theo_Log-Lip}.
	   
\begin{proof}[Proof of Theorem \ref{Theo_Log-Lip}]
We prove the theorem for $x_0\in B_1^0$ fixed, though arbitrary. For simplicity, we let $x_0\equiv 0$. We split the argument into three steps.

\medskip

\noindent{\bf Step 1 - }We prove the existence of a universal constant $0<\rho\ll1$ and of sequences $(a_n)_{n\in\mathbb{N}}\subset\mathbb{R}$ and $(M_n)_{n\in\mathbb{N}}\subset \mathbb{R}^{d\times d}$ such that
\begin{equation}\label{eq_redlion}
	\left|a_n-a_{n+1}\right|\leq C\rho^n,
\end{equation} 
\begin{equation}\label{eq_headoftheriver}
	\left|M_n-M_{n+1}\right|\leq C,
\end{equation} 
and
\begin{equation}\label{eq_judetheobscure}
	\sup_{x\in B_{\rho^n}^+}\left|u(0)-a_nx_d-M_n^{i,j}x_ix_j\right|\leq \rho^{2n},
\end{equation}
for every $n\in\mathbb{N}$. As usual, we resort to an induction argument. The next step accounts for the base case.

\medskip

\noindent{\bf Step 2 - }Set $a_0\equiv 0$ and $M_0\equiv 0$. Let $h$ be the function from Corollary \ref{cor_approx_2}, satisfying a $\tau$-proximity regime with respect to $u$ in $B_{3/4}^+$. Because of the boundary regularity available for $h$, we infer there exists $C>0$ such that 
\begin{equation}\label{eq_tree}
	\sup_{x\in B_\rho^+}\left|h(0)-Dh(0)\cdot x-\frac{1}{2}D^2h(0)x\cdot x\right|\leq C\rho^{2+\alpha},
\end{equation}
where $\alpha\in (0,1)$ depends only on the dimension, $\lambda$ and $\Lambda$. Notice the only non-trivial coordinate in $Dh(0)$ is the $d$-th one, since $h\equiv 0$ on $B_{3/4}^0$. Now, Corollary \ref{cor_approx_2} builds upon \eqref{eq_tree} in the usual way to ensure
\[
	\sup_{x\in B_\rho^+}\left|u(x)-Dh(0)\cdot x+\frac{1}{2}D^2h(0)x\cdot x\right|\leq C\tau+C\rho^{2+\alpha}.
\]
Define $\tau$ and $\rho$ as
\[
	\tau:=\left(\frac{\rho^2}{2C}\right)\hspace{.3in}\mbox{and}\hspace{.3in}\rho:=\left(\frac{1}{2C}\right)^\frac{1}{\alpha}.
\]
Take $a_1=Dh(0)\cdot e_d$ and $2M_1:=D^2h(0)$; the former computation ensures \eqref{eq_judetheobscure}. The universal estimates available for $h$ ensure \eqref{eq_redlion}-\eqref{eq_headoftheriver} also hold, and the base case follows. The next step concludes the induction argument.

\medskip

\noindent{\bf Step 3 - }Suppose we have established \eqref{eq_redlion}-\eqref{eq_judetheobscure} for $n=1,\ldots,k$. Define $u_k:B_1^+\to\mathbb{R}$ as
\[
	u_k(x):=\frac{u(\rho^kx)-a_k\rho^kx-\rho^{2k}M_kx\cdot x}{\rho^{2k}}.
\]
Notice that $u_k$ satisfies
\[
	\min\left(F_1^k(D^2u_k,Du_k,x),F_2^k(D^2u_k,Du_k,x)\right)\leq \left|f\right|\hspace{.2in}\mbox{in}\hspace{.2in}B_1^+
\]
and 
\[
	\max\left(F_1^k(D^2u_k,Du_k,x),F_2^k(D^2u_k,Du_k,x)\right)\geq -\left|f\right|\hspace{.2in}\mbox{in}\hspace{.2in}B_1^+,
\]
where
\[
	F_i^k(M,p,x):=F_i(M+M_k,\rho^kp+a_k e_d+\rho^kM_k\cdot x,x).
\]

Since Assumption \ref{assump_F-kappa} holds uniformly for $F_i$, it also holds for $F_i^k$. Therefore, the argument in the previous step ensures the existence of a real number $\tilde a_k$ and a matrix $\tilde M_k$ such that 
\[
	\sup_{x\in B_\rho^+}\left|u_k(x)-\tilde a_kx_d-\frac{1}{2}\tilde M_kx\cdot x \right|\leq \rho^2.
\]
Scaling back to the unit picture, we obtain
\begin{equation}\label{eq_challah}
	\sup_{x\in B_{\rho^{k+1}}}\left|u(x)-a_{k+1}x_d-\frac{1}{2}M_{k+1}x\cdot x\right|\leq \rho^{2(k+1)},
\end{equation}
where
\[
	a_{k+1}:=a_k+\rho^k\tilde a_k\hspace{.3in}\mbox{and}\hspace{.2in}M_{k+1}:=M_k+\tilde M_k.
\]
Gathering \eqref{eq_challah} and the definitions of $a_{k+1}$ and $M_{k+1}$, one verifies \eqref{eq_redlion}-\eqref{eq_judetheobscure} at the level $k+1$ and completes the induction argument. Once \eqref{eq_redlion}-\eqref{eq_judetheobscure} are available, the proof of the theorem follows from straightforward computations (e.g., \cite[p. 1398]{PimentelUrbano}).
\end{proof}

\section{Sobolev estimates up to the boundary}\label{sec_GRU}

In the sequel, we present the proof of Theorem \ref{thm_Sobolev}. In line with Remark \ref{rem_scalingfromhell}, we set $\Omega:=B_{14\sqrt{d}}^+$ throughout this section. We start with a boundary variant of Lin's $L^\mu$-estimate for the solutions to \eqref{eq_viscmin}-\eqref{eq_bdry}, as established in \cite[Proposition 2.12]{WinterNiki}. See also \cite{Lin1986}, for the original argument in the linear case, and \cite[Proposition 7.4]{Caffarelli_Cabre1995} for the fully nonlinear elliptic setting.

\begin{proposition}\label{Pro:W2delta_bound}
Let $u\in C(B_{14\sqrt{d}}^+)$ be a viscosity solution to \eqref{eq_viscmin}-\eqref{eq_bdry}. Suppose Assumptions A\ref{assump_domain} and A\ref{assump_Felliptic} are in force. Suppose further that $f\in L^p(B_{14\sqrt{d}}^+)$ . Then there exist universal constants $\mu>0$ and $C>0$ such that 
\[
	\int_{B_{13\sqrt{d}}^+}\left|D^2u\right|^\mu{\rm d}x\leq C \left(\left\|u\right\|_{L^\infty\left(B_{14\sqrt{d}}^+\right)}+\left\|f\right\|_{L^d\left(B_{14\sqrt{d}}^+\right)}\right).
\]
\end{proposition}

The proof of Proposition \ref{Pro:W2delta_bound} follows from the very same argument as in the proof of \cite[Proposition 2.12]{WinterNiki}. We note the former result is stated in \cite{WinterNiki} for the class $S(f)$. However, a close inspection of the proof reveals that it holds for the class $S^*(f)$. We notice Proposition \ref{Pro:W2delta_bound} implies an upper bound for the measure of the sets $A_t$, in terms of $t>0$. Indeed, it unlocks a decrease rate for the measure of those sets. This is the content of the next proposition.

\begin{proposition}\label{Pro:W2delta_Winter}
Let $u\in S^*(f)$ in $B_{12\sqrt{d}}^+\subset \Omega\subset \Rr^d_+$ and suppose $u$ is normalized. Suppose Assumptions A\ref{assump_domain}, A\ref{assump_Felliptic}, A\ref{assump_F-kappa} and A\ref{assump_FC1} are in force. Then there exist universal constants $C, \mu>0$ such that, if $\|f\|_{L^p(\Omega)}\leq 1$, we get 
\begin{equation}\label{eq_olesteen}
	\left|A_t(u,\Omega)\cap ((Q_1^+)+x_0)\right|\leq C t^{-\mu}
\end{equation}
for any $x_0\in B_{9\sqrt{d}}^+$.
\end{proposition}

As usual, once Proposition \ref{Pro:W2delta_Winter} is available, we resort to an approximation lemma (Proposition \ref{prop_approx_w2p}) to improve the decay rate in \eqref{eq_olesteen}. We start by refining the bound on $|G_M(u)|$.
\begin{lemma}\label{Lm:first_it_w2p}
Let $u\in C(B_{14\sqrt{d}}^+)$ be a viscosity solution to \eqref{eq_viscmin}-\eqref{eq_bdry}. Let $\sigma\in (0,1)$ be fixed, though arbitrary. Suppose Assumptions A\ref{assump_domain}, A\ref{assump_Felliptic}, A\ref{assump_F-kappa} and A\ref{assump_FC1} are in force. Suppose further $f\in L^p(B_{14\sqrt{d}}^+)$. If there exists $x_0\in B_{9\sqrt{d}}^+$ such that 
\[
	G_1(u,\Omega)\cap \left(Q_2^++x_0\right)\neq \emptyset,
\]
then one can find a universal constant $M>1$ such that
\[
	\left|G_M(u,\Omega)\cap \left(Q_1^++x\right)\right|\geq 1-\sigma,
\]
for every $x\in B_{9\sqrt{d}}^+$.
\end{lemma}
\begin{proof}
Let 
\[
\tilde{u}=\frac{u-L}{C}
\]
where $L=a+b\cdot x$ is an affine function and $C$ is taken sufficiently large such that $\|\tilde{u}\|_{L^\infty}\leq 1$ and
\[
-|x|^2\leq \tilde{u}(x)\leq |x|^2 \quad \mbox{ in } B_{14\sqrt{d}}^+\setminus B_{12\sqrt{d}}^+.
\]
Let $\tilde F_i(M,p,x)=C^{-1}F_i(CM,Cp+Cb,x)$, $\tilde F(M,p)=C^{-1} F(CM,Cp+Cb)$ and $\tilde f=C^{-1}f$. The rescaled data satisfy Assumptions A\ref{assump_Felliptic}, A\ref{assump_F-kappa} and A\ref{assump_FC1}, with the same constants. Also, $\tilde{u}$ solves the inequalities in \eqref{eq_viscmin} driven by the rescaled data.

As in the proof of Proposition \ref{prop_approx_w2p}, let $h$ be the solution of
\begin{equation*}
	\begin{cases}
		F(D^2h,Dh)=0 & \mbox{ in } B_{13\sqrt{d}}^+\\
		h=\tilde u & \mbox{ on } \partial B_{13\sqrt{d}}^+.
	\end{cases}
\end{equation*}
The maximum principle ensures that 
\[
\left\|h\right\|_{L^\infty\left(B_{13\sqrt{d}}^+\right)}\leq \left\|\tilde u\right\|_{L^\infty\left(B_{13\sqrt{d}}^+\right)}\leq 1.
\]
An application of Proposition \ref{prop_C2alpha_SS} ensures the existence of a universal constant $C_1>0$ such that 
\[
	\left\|h\right\|_{C^{2,\alpha}\left(B_{12\sqrt{d}}^+\right)}\leq C_1.
\]
Hence,
\[
	A_N\left(h,B_{12\sqrt{d}}^+\right)=\emptyset,
\]
for  $N:=2C_1$. Extend $h$ continuously such that $h=\tilde u$ outside $B_{13\sqrt{d}}^+$ and 
\[
	\left\|\tilde u-h\right\|_{L^\infty\left(B_{14\sqrt{d}}^+\right)}=\left\|\tilde u-h\right\|_{L^\infty\left(B_{12\sqrt{d}}^+\right)}.
\] 
Hence $\left\|\tilde u-h\right\|_{L^\infty(\Omega)}\leq 2$, and
\[
	-2-|x|^2\leq h(x)\leq 2+|x|^2\quad \mbox{ in } \Omega\setminus B_{12\sqrt{d}}^+.
\]
These estimates imply
\[
	A_{M_0}(h,\Omega)\cap (Q_1^+)=\emptyset
\]
for some universal $M_0>0$. Set $w:=\tilde u-h$ and resort to Proposition \ref{prop_approx_w2p} to conclude
\[
	\left\|w\right\|_{L^\infty\left(B_{12\sqrt{d}}^+\right)}+\left\|\varphi\right\|_{L^d\left(B_{12\sqrt{d}}^+\right)}\leq C\left(\kappa^\gamma+\left\|g\right\|_{L^d}\right)\leq C\kappa^\gamma,
\]
with $w\in S^*(\varphi)$ in $B_{12\sqrt{d}}^+$. It follows that
\[
	\left\|w\right\|_{L^\infty\left(B_{14\sqrt{d}}^+\right)}\leq \left\|w\right\|_{L^\infty\left(B_{12\sqrt{d}}^+\right)}\leq C\kappa^\gamma.
\]
Finally, observe that $\tilde w= (C\kappa^\gamma)^{-1}w$ satisfies the assumptions of Proposition \ref{Pro:W2delta_Winter}; then we obtain, for $t>1$,
\[
	\left|A_t(\tilde{w}, B_{14\sqrt{d}}^+)\cap Q_1^+\right|\leq Ct^{-\mu},
\]
where $C>0$ is a universal constant. Noting that $A_{2M_0}(\tilde u)\subset \left(A_{M_0}(w)\cup A_{M_0}(h)\right)$, and using the scaling properties of $A_t$, we get
\[
	\left|A_{2M_0}(\tilde u, B_{14\sqrt{d}}^+)\cap Q_1^+\right|\leq C\kappa^{\mu \gamma}M_0^\mu.
\]
To complete the proof, notice $A_{2M_0}(\tilde u,B_{14\sqrt{d}}^+) = A_{2CM_0}(u,B_{14\sqrt{d}}^+)$, let $M=2CM_0$ and choose $0<\kappa\ll1$ sufficiently small.
\end{proof}

Once again, for the sake of completeness, we recall the consequence of the so-called Calder\'on-Zygmund cube decomposition. Let $Q_1$ be the unit cube and split it dyadically and successively, calling $\tilde{Q}$ a predecessor of $Q$ if we obtain $Q$ from the splitting of $\tilde{Q}$. The following lemma is pivotal in our analysis.

\begin{lemma}[Corollary to the Calder\'on-Zygmund cube decomposition]\label{Lm:CZ_cube}
Let $A\subset B\subset Q_1$ be measurable sets. Let $0<\sigma<1$ be such that
\begin{enumerate}
\item $|A|\leq \delta$;
\item If $Q$ is a dyadic cube such that $|A\cap Q|>\sigma|Q|$, then $\tilde Q\subset B$.
\end{enumerate}
Then $|A|\leq \sigma |B|$.
\end{lemma}
For the proof of Lemma \ref{Lm:CZ_cube}, we refer the reader to \cite[Lemma 4.2]{Caffarelli_Cabre1995}.
\begin{proposition}\label{Lm:iterations_w2p}
Let $u\in C(B_{14\sqrt{d}}^+)$ be a viscosity solution to \eqref{eq_viscmin}-\eqref{eq_bdry}. Let $\sigma\in (0,1)$ be fixed, though arbitrary. Suppose Assumptions A\ref{assump_domain}, A\ref{assump_Felliptic}, A\ref{assump_F-kappa} and A\ref{assump_FC1} are in force. Suppose further $f\in L^p(B_{14\sqrt{d}}^+)$. Set
\begin{equation*}
	\begin{split}
		A:=&\, A_{M^{k+1}}\left(u,B_{14\sqrt{d}}^+\right)\cap Q_1^+\\
		B:=&\, (A_{M^{k+1}}\left(u,B_{14\sqrt{d}}^+\right)\cap Q_1^+)\cup \left\lbrace x\in Q_1^+ \,:\, M\left(f^d\right)\geq \left(C_0M^k\right)^d\right\rbrace.
	\end{split}
\end{equation*}
Then $|A|\leq \sigma|B|$, where $M>1$ is universal.
\end{proposition}
\begin{proof}
We resort to Lemma \ref{Lm:CZ_cube}. Notice $A\subset B\subset Q_1^+$ and, from Lemma \ref{Lm:first_it_w2p}, we conclude $|A|\leq \sigma<1$. 

We aim to establish that dyadic cubes with $|A\cap Q|>\sigma|Q|$ must satisfy $\tilde{Q}\subset B$. Consider a generic dyadic cube $Q=Q_{1/2^i}^++x_0$ for some $i\in\mathbb{N}$ and $x_0$, and consider its predecessor $\tilde{Q}=Q_{1/2^{i-1}}^++\tilde{x_0}$.

Assume that $Q$ satisfies
\begin{equation}\label{eq:w2p_it_contr}
	\left|A\cap Q\right|=\left|A_{M^{k+1}}\left(u,B_{14\sqrt{d}}\right)\cap Q\right|>\sigma|Q|,
\end{equation}
however $\tilde{Q}\not \subset B$. In this case, there exists $x_1\in \tilde Q\setminus B$. It means that
\begin{equation}\label{eq:w2p_B}
	x_1\in \tilde{Q}\cap G_{M^k}\left(u,B_{14\sqrt{d}}^+\right)\hspace{.2in}\mbox{and}\hspace{.2in} M\left(f^d\right)(x_1)<\left(C_0M^k\right)^d.
\end{equation}
Now, we distinguish two cases depending on the distance from $x_0$ to $\left\lbrace x_d=0\right\rbrace$.

If $|x_0-x_0'|<\frac{8}{2^i}\sqrt{d}$, we consider $L(y)=x_0'+2^{-i}y$ and define 
\[
	\tilde u(y):=\frac{2^{2i}}{M^{k}}u(L(y)),
\]
\[
	\tilde F_i(X,p,y):=\frac{1}{M^{k}}F_i\left(M^kX,\frac{2^i}{M^{k}}p,L(y)\right),
\]
\[
	\tilde F(X,p):=\frac{1}{M^{k}}F\left(M^kX,\frac{2^i}{M^{k}}p\right),
\]
and
\[
	\tilde f(y):=\frac{1}{M^{k}}f\left(L(y)\right)
\]

Since $Q\subset Q_1^+$ implies $B_{14\sqrt{d}/2^i}(x_0')\subset B_{14\sqrt{d}}^+$, we have that $\tilde u$ solves
\begin{align*}
\begin{cases}
\min\left(\tilde F_1(D^2\tilde u,D\tilde u,x),\tilde F_2(D^2\tilde u,D\tilde u,x)\right)\leq |\tilde f(x)| & \mbox{ in } B_{14\sqrt{d}}^+\\
\max\left(\tilde F_1(D^2\tilde u,D\tilde u,x),\tilde F_2(D^2\tilde u,D\tilde u,x)\right)\geq -|\tilde f(x)| & \mbox{ in } B_{14\sqrt{d}}^+,
\end{cases}
\end{align*}
with $\tilde u=0  \mbox{ on } B_{14\sqrt{d}}'$, and the same assumptions as before.

Note that from \eqref{eq:w2p_B} we have
\begin{align*}
\left\|\tilde f\right\|_{L^d(B^+_{12\sqrt{d}})}\leq&\, \frac{2^i}{M^k}\left( \int_{Q_{42\sqrt{d}/2^i}(x_1)} f^d(x)\, dx\right)^\frac{1}{d}
\leq  C(d)C_0\leq \kappa
\end{align*}
for $C_0$ sufficiently small. Here we used the fact that $B^+_{12\sqrt{d}/2^i}(x_0')\subset Q_{42\sqrt{d}/2^i}(x_1)$.

From \eqref{eq:w2p_B} it also follows that 
\[
G_1(\tilde u, L^{-1}(B^+_{14\sqrt{d}}))\cap (Q_2^++2^i(\tilde{x_0}-x_0'))\neq \emptyset
\]
From $|x_0-\tilde{x_0}|\leq 2^{-2}\sqrt{d}$ we get $|2^i(\tilde{x_0}-x_0')|<9\sqrt{d}$. Hence, the assumptions of Lemma \ref{Lm:first_it_w2p} are in force. Therefore,
\[
	\left|G_M\left(\tilde u, L^{-1}(B^+_{14\sqrt{d}})\right)\cap \left(Q_1^++2^i(x_0-x_0')\right)\right|\geq 1-\sigma.
\]
Rescaling back to $u$, we get
\[
	\left|G_{M^k}\left(u,B^+_{14\sqrt{d}}\right)\cap Q\right|\geq (1-\sigma)\left|Q\right|,
\]
which contradicts \eqref{eq:w2p_it_contr}.

On the other hand, if $|x_0-x_0'|\geq \frac{8}{2^i}\sqrt{d}$ we conclude
\[
B^+_{8\sqrt{d}/2^i}(x_0+2^{-i-1}e_d)\subset B^+_{8\sqrt{d}}.
\]
Now, we use instead the transformation $L(y)=(x_0+2^{-i-1}e_d)+2^{-i}y$. Proceeding exactly as in the first case, we evoke \cite[Lemma 7.11]{Caffarelli_Cabre1995} instead of Lemma \ref{Lm:first_it_w2p} and complete the proof.
\end{proof}

\begin{proof}[Proof of Theorem \ref{thm_Sobolev}]
Let $0<\kappa\ll1$ be the same as Lemma \ref{Lm:first_it_w2p}. Fix $\varepsilon_0>0$ to be set further. Let $M>0$ and $C_0>0$ be the same as in Lemma \ref{Lm:iterations_w2p}, and choose $\sigma:=(2M^p)^{-1}$. We define
\begin{align*}
\alpha_k:=&\,\left|A_{M^k}\left(\tilde u, B^+_{14\sqrt{d}}\right)\cap Q_1^+\right|,\\
\beta_k:=&\,\left|\left\lbrace x\in Q_1^+\, |\, M\left(f^d\right)(x)\geq \left(C_0M^k\right)^d\right\rbrace\right|.
\end{align*}
An application of Proposition \ref{Lm:iterations_w2p} yields $\alpha_{k+1}\leq \sigma(\alpha_k+\beta_k)$; thus
\begin{align}\label{eq:w2p_alphak}
\alpha_k\leq \varepsilon_0^k+\sum_{i=0}^{k-1}\sigma^{k-i}\beta_i.
\end{align}
Using the fact that $M(f^d)$ is of strong-type $(p,p)$ for $1<p\leq \infty$, and the integrability of $f$, we get
\begin{align*}
	\left\|M\left(f^d\right)\right\|_{L^\frac{p}{d}(\Omega)}\leq&\, C(d,p)\left\|f^d\right\|_{L^\frac{p}{d}(\Omega)}=C(d,p)\left\|f\right\|_{L^p(\Omega)}^d\leq C(d,p)
\end{align*}
By Lemma \ref{lemma_distrLp} we get
\begin{align}\label{eq:w2p_bkMk}
\sum_{k\geq 0} M^{pk}\beta_k\leq C(d,p).
\end{align}
The choice of $\sigma$, together with \eqref{eq:w2p_alphak} and \eqref{eq:w2p_bkMk}, implies
\begin{align*}
\sum_{k\geq 0}M^{pk}\alpha_k\leq \sum_{k\geq 0}2^{-k}+\left(\sum_{k\geq 0}M^{pk}\beta_k\right)\left(\sum_{k\geq 0}M^{pk}\tau^k\right)\leq C(d,p)
\end{align*} 
By Lemma \ref{lemma_distrLp} again, we conclude
\[
\left\|D^2u\right\|_{L^p\left(B^+_{12\sqrt{d}}\right)}\leq C(d,p,M)
\]
as intended, and the proof is complete.
\end{proof}

\begin{remark}[${\rm BMO}$-estimates and $C^{1,{\rm Log-Lip}}$-regularity]\label{rem_bodleian}\normalfont
In case $p=\infty$, our arguments imply that $u\in W^{2,{\rm BMO}}(B_{1/2}^+)$. As a consequence, $u\in C^{1,{\rm Log-Lip}}(\overline{B_{1/2}^+})$, with estimates; see, for instance \cite{Cia}. It means that for every $x_0\in\overline{B_{1/2}^+}$ there exists $C>0$ such that 
	\[
		\sup_{x\in B_r(x_0)}\left|u(x)-u(x_0)-Du(x_0)\cdot(x-x_0)\right|\leq Cr^2\ln\frac{1}{\left|x-x_0\right|},
	\]
and the constant does not depend on the distance ${\rm dist}(x_0,\partial\Omega)$.
We notice that imposing a differentiability condition on the limiting profile $F$ improves the regularity available under merely uniform ellipticity conditions on the limiting operator. See \cite{Savin2007}. 
\end{remark}

\begin{remark}[Regularity in $W^{1,q}$-spaces and potential estimates]\label{rem_cairns}\normalfont
If $d/2\leq p_0<p<d$, it is well-known that $W^{2,p}$-regularity estimates are no longer available for the solutions to \eqref{eq_viscmin}. However, Proposition \ref{prop_approx_w2p} builds upon the arguments in \cite{Swiech1997}, under minor adjustments, to yield estimates in $W^{1,q}$-spaces. Here, the exponent $q>1$ can be chosen arbitrarily under the constraint
\[
	q<p^*:=\frac{dp}{d-p},
\]
with $d^*:=\infty$. 
To obtain estimates in the borderline case $q=\infty$, one may resort to the machinery of potential estimates, as put forward in \cite{DasKuuMin2014}.
\end{remark}

\bigskip

\noindent{\bf Acknowledgements:} The authors are grateful to Diego R. Moreira for his interest in the material in this paper and fruitful discussions. DJ is supported by PRIN 2022 7HX33Z - CUP J53D23003610006: Pattern formation in nonlinear phenomena. EP and DS are supported by the Centre for Mathematics of the University of Coimbra (CMUC, funded by the Portuguese Government through FCT/MCTES, DOI 10.54499/UIDB/00324/2020).

\bigskip

\bibliographystyle{plain}
\bibliography{my_bibliography}

\begin{thebibliography}{10}

\bibitem{Bao-Li-Yin1}
Ellen~Shiting Bao, Yan~Yan Li, and Biao Yin.
\newblock Gradient estimates for the perfect conductivity problem.
\newblock {\em Arch. Ration. Mech. Anal.}, 193(1):195--226, 2009.

\bibitem{Bao-li-Yin2}
Ellen~Shiting Bao, Yan~Yan Li, and Biao Yin.
\newblock Gradient estimates for the perfect and insulated conductivity
  problems with multiple inclusions.
\newblock {\em Comm. Partial Differential Equations}, 35(11):1982--2006, 2010.

\bibitem{Caffarelli_Cabre1995}
Luis~A. Caffarelli and Xavier Cabr\'e.
\newblock {\em Fully nonlinear elliptic equations}, volume~43 of {\em American
  Mathematical Society Colloquium Publications}.
\newblock American Mathematical Society, Providence, RI, 1995.

\bibitem{CSCS2021}
Luis~A. Caffarelli, Mar\'ia Soria-Carro, and Pablo~Ra\'ul Stinga.
\newblock Regularity for {$C^{1,\alpha}$} interface transmission problems.
\newblock {\em Arch. Ration. Mech. Anal.}, 240(1):265--294, 2021.

\bibitem{Campanato1957}
Sergio Campanato.
\newblock Sul problema di {M}. {P}icone relativo all'equilibrio di un corpo
  elastico incastrato.
\newblock {\em Ricerche Mat.}, 6:125--149, 1957.

\bibitem{Campanato1959}
Sergio Campanato.
\newblock Sui problemi al contorno per sistemi di equazioni differenziali
  lineari del tipo dell'elasticit\`a. {I}.
\newblock {\em Ann. Scuola Norm. Sup. Pisa Cl. Sci. (3)}, 13:223--258, 1959.

\bibitem{Campanato1959a}
Sergio Campanato.
\newblock Sui problemi al contorno per sistemi di equazioni differenziali
  lineari del tipo dell'elasticit\`a. {II}.
\newblock {\em Ann. Scuola Norm. Sup. Pisa Cl. Sci. (3)}, 13:275--302, 1959.

\bibitem{Cia}
Andrea Cianchi.
\newblock Continuity properties of functions from {Orlicz}-{Sobolev} spaces and
  embedding theorems.
\newblock {\em Ann. Sc. Norm. Super. Pisa, Cl. Sci., IV. Ser.}, 23(3):575--608,
  1996.

\bibitem{CIL}
Michael~G. Crandall, Hitoshi Ishii, and Pierre-Louis Lions.
\newblock User's guide to viscosity solutions of second order partial
  differential equations.
\newblock {\em Bull. Amer. Math. Soc. (N.S.)}, 27(1):1--67, 1992.

\bibitem{DasKuuMin2014}
Panagiota Daskalopoulos, Tuomo Kuusi, and Giuseppe Mingione.
\newblock Borderline estimates for fully nonlinear elliptic equations.
\newblock {\em Commun. Partial Differ. Equations}, 39(3):574--590, 2014.

\bibitem{CdF}
Cristiana De~Filippis.
\newblock Fully nonlinear free transmission problems with nonhomogeneous
  degeneracies.
\newblock {\em Interfaces Free Bound.}, 24(2):197--233, 2022.

\bibitem{DJ2}
Davide Giovagnoli and David Jesus.
\newblock A fully nonlinear transmission problem degenerating on the interface.
\newblock Preprint, {arXiv}:2410.16957 [math.{AP}] (2024), 2024.

\bibitem{Huaroto-Pimentel-Rampasso-Swiech}
Gerardo Huaroto, Edgard~A. Pimentel, Giane~C. Rampasso, and Andrzej
  \'Swi\k{e}ch.
\newblock A fully nonlinear degenerate free transmission problem.
\newblock {\em Ann. PDE}, 10(1):Paper No. 5, 30, 2024.

\bibitem{DJ1}
David Jesus.
\newblock A degenerate fully nonlinear free transmission problem with variable
  exponents.
\newblock {\em Calc. Var. Partial Differential Equations}, 61(1):Paper No. 29,
  21, 2022.

\bibitem{KoikeSwiech2022}
Shigeaki Koike and Andrzej {\'S}wi{\k{e}}ch.
\newblock Aleksandrov-{B}akelman-{P}ucci maximum principle for
  {$L^p$}-viscosity solutions of equations with unbounded terms.
\newblock {\em J. Math. Pures Appl. (9)}, 168:192--212, 2022.

\bibitem{Li-Nirenberg2003}
Yanyan Li and Louis Nirenberg.
\newblock Estimates for elliptic systems from composite material.
\newblock {\em Comm. Pure Appl. Math.}, 56(7):892--925, 2003.
\newblock Dedicated to the memory of J\"urgen K. Moser.

\bibitem{Li-Vogelius2000}
Yanyan Li and Michael Vogelius.
\newblock Gradient estimates for solutions to divergence form elliptic
  equations with discontinuous coefficients.
\newblock {\em Arch. Ration. Mech. Anal.}, 153(2):91--151, 2000.

\bibitem{Lin1986}
Fanghua Lin.
\newblock Second derivative {{\(L^ p\)}}-estimates for elliptic equations of
  nondivergent type.
\newblock {\em Proc. Am. Math. Soc.}, 96:447--451, 1986.

\bibitem{Lions1956}
J.~L. Lions.
\newblock Contribution \`a{} un probl\`eme de {M}. {M}. {P}icone.
\newblock {\em Ann. Mat. Pura Appl. (4)}, 41:201--219, 1956.

\bibitem{Picone1954}
Mauro Picone.
\newblock Sur un probl{\`e}me nouveau pour l'{\'e}quation lin{\'e}aire aux
  d{\'e}riv{\'e}es partielles de la th{\'e}orie math{\'e}matique classique de
  l'{\'e}lasticit{\'e}.
\newblock {\em Colloque sur les {\'e}quations aux d{\'e}riv{\'e}es partielles,
  CBRM, Bruxelles}, pages 9--11, 1954.

\bibitem{Pimentel-Santos2023}
Edgard~A. Pimentel and Makson~S. Santos.
\newblock Fully nonlinear free transmission problems.
\newblock {\em Interfaces Free Bound.}, 25(3):325--342, 2023.

\bibitem{PimSanTei2022}
Edgard~A. Pimentel, Makson~S. Santos, and Eduardo~V. Teixeira.
\newblock Fractional {Sobolev} regularity for fully nonlinear elliptic
  equations.
\newblock {\em Commun. Partial Differ. Equations}, 47(8):1539--1558, 2022.

\bibitem{Pimentel-Swiech2022}
Edgard~A. Pimentel and Andrzej \'Swi\k{e}ch.
\newblock Existence of solutions to a fully nonlinear free transmission
  problem.
\newblock {\em J. Differential Equations}, 320:49--63, 2022.

\bibitem{PimentelUrbano}
Edgard~A. Pimentel and Jos\'e{}~Miguel Urbano.
\newblock Existence and improved regularity for a nonlinear system with
  collapsing ellipticity.
\newblock {\em Ann. Sc. Norm. Super. Pisa Cl. Sci. (5)}, 22(3):1385--1400,
  2021.

\bibitem{Savin2007}
Ovidiu Savin.
\newblock Small perturbation solutions for elliptic equations.
\newblock {\em Commun. Partial Differ. Equations}, 32(4):557--578, 2007.

\bibitem{Schechter1960}
Martin Schechter.
\newblock A generalization of the problem of transmission.
\newblock {\em Ann. Scuola Norm. Sup. Pisa Cl. Sci. (3)}, 14:207--236, 1960.

\bibitem{SilSir_2014}
Luis Silvestre and Boyan Sirakov.
\newblock Boundary regularity for viscosity solutions of fully nonlinear
  elliptic equations.
\newblock {\em Commun. Partial Differ. Equations}, 39(9):1694--1717, 2014.

\bibitem{SoriaCarro-Stinga}
Mar\'ia Soria-Carro and Pablo~Ra\'ul Stinga.
\newblock Regularity of viscosity solutions to fully nonlinear elliptic
  transmission problems.
\newblock {\em Adv. Math.}, 435:Paper No. 109353, 52, 2023.

\bibitem{Stampacchia1956}
Guido Stampacchia.
\newblock Su un problema relativo alle equazioni di tipo ellittico del secondo
  ordine.
\newblock {\em Ricerche Mat.}, 5:3--24, 1956.

\bibitem{Swiech1997}
Andrzej \'Swi\k{e}ch.
\newblock {$W^{1,p}$}-interior estimates for solutions of fully nonlinear,
  uniformly elliptic equations.
\newblock {\em Adv. Differential Equations}, 2(6):1005--1027, 1997.

\bibitem{WinterNiki}
Niki Winter.
\newblock {$W^{2,p}$} and {$W^{1,p}$}-estimates at the boundary for solutions
  of fully nonlinear, uniformly elliptic equations.
\newblock {\em Z. Anal. Anwend.}, 28(2):129--164, 2009.

\end{thebibliography}

\bigskip

\noindent\textsc{David Jesus}\\
Universit\`a di Bologna
\\Bologna, Piazza di Porta San Donato 5
\\40126, Italy\\
\noindent\texttt{david.jesus2@unibo.it}

\bigskip

\noindent\textsc{Edgard A. Pimentel (Corresponding author)}\\
University of Coimbra\\
CMUC, Department of Mathematics\\ 
3000-143 Coimbra, Portugal\\
\noindent\texttt{edgard.pimentel@mat.uc.pt}

\bigskip

\noindent\textsc{David Stolnicki}\\
University of Coimbra\\
CMUC, Department of Mathematics\\ 
3000-143 Coimbra, Portugal\\
\noindent\texttt{david.stolnicki@mat.uc.pt}

\end{document}